\documentclass[a4paper,12pt,reqno]{amsart}
\usepackage{lmodern,amsmath,amssymb,mathtools,amsthm}
\usepackage{enumerate}
\usepackage{graphics,graphicx}
\usepackage{bigints}
\usepackage{geometry}
\usepackage{wasysym}
\usepackage{cleveref}
\usepackage{floatrow,morefloats}
\usepackage{epsfig}
\usepackage{epstopdf}
\usepackage{caption,subcaption}%
\numberwithin{equation}{section}
\newtheorem{theorem}{Theorem}[section]
\newtheorem{lemma}{Lemma}[section]
\newtheorem{proposition}{Proposition}[section]
\newtheorem{corollary}{Corollary}[section]
\newtheorem{remark}{Remark}[section]
\newtheorem{definition}{Definition}[section]

\theoremstyle{definition}
\newtheorem{example}{Example}

\newtheorem{condition}{Condition}
\begin{document}
\bibliographystyle{amsplain}
\title{{{
Generalized Mittag-Leffler stability of Hilfer fractional impulsive differential systems
}}}
\author{Divya Raghavan
}
\address{
Department of Mathematics,
Indian Institute of Technology, Roorkee-247667,
Uttarakhand, India
}
\email{divyar@ma.iitr.ac.in, madhanginathan@gmail.com}
\author{
Sukavanam Nagarajan
}
\address{
Department of  Mathematics  \\
Indian Institute of Technology, Roorkee-247 667,
Uttarkhand, India
}
\email{n.sukavanam@ma.iitr.ac.in}
%%%
\author{
Chengbo Zhai
}
\address{
School of Mathematical Sciences
Shanxi University,
Taiyuan 030006 SHANX
PEOPLES REPUBLIC OF CHINA}
\email{cbzhai@sxu.edu.cn}

\bigskip
\begin{abstract}
This paper establishes integral representations of mild solutions of impulsive Hilfer fractional differential equations with impulsive conditions and fluctuating lower bounds at impulsive points. Further, the paper provides sufficient conditions for generalized Mittag-Leffler stability of a class of impulsive fractional differential systems with Hilfer order. The analysis extends through both, instantaneous and non-instantaneous impulsive conditions. The theory utilizes continuous Lyapunov functions, to ascertain the stability conditions. An example is provided to study the solution of the system with changeable lower bound for the non-instantaneous impulsive conditions.
\end{abstract}
\maketitle
%
%
%\pagestyle{myheadings}
%\markboth
%{Divya Raghavan and N.Sukavanam}
%{Criteria for Generalized Mittag-Leffler stability of impulsive differential system with Hilfer fractional order }
%
%
\section{Introduction}
Many biological happenings sustain perturbations for a period of time. Some of such perturbations may persist for a very short span, may be only at certain points or it may stretch to a finite time interval. In accordance with the duration of perturbation, systems are branched as instantaneous impulsive
system and non-instantaneous impulsive system respectively. The instantaneous impulsive system finds its application in models where the system changes its constraint suddenly. For example, the sudden change of speed/ direction of a moving car or in the change of trajectory of ball bouncing on a hard surface are modeled as instantaneous impulsive systems. For applications of instantaneous impulsive systems, one can refer to \cite{Impulsive-instant-book} by Stamova and Stamov, where, various models involving impulsive conditions are discussed in detail. On the other hand, non-instantaneous impulsive systems gets involved when the perturbations are not negligible. Such a situation was first analyzed by Hern\'{a}ndez\ and\  O'Regan in \cite{non-instant}, where they proposed a new impulsive conditions where the perturbations prolong for a finite interval of time and not just at some fixed moments. Precisely, instead of impulsive points $t_{k}$, they considered the finite time interval $[t_{k},p_{k}]$ where the perturbation occurs. For instance, a differential equation model that reveals the impact of the drug in a body for a certain period of time admits a non-instantaneous impulse. Besides the field of medicine, most of the real-life problems including the study of geographical conditions, to estimate the impact of global warming, in varied field of physics,  involve non-instantaneous impulses. The detailed theory and application regarding non-instantaneous impulses is available in whole lot in the literature, see for instance, the book by
Agarwal\textit{et al}. \cite{Impulsive-non-instant-book}. Numerous research articles constantly emerge that deal with the above two impulsive conditions.

In both the impulsive systems, in particular, the state of the system keeps on varying. Thus, stability is one such property which has to be
addressed, as it determines the stable region of the state of the system. As an example, a power system uses stability analysis to prove if it has the ability to withstand the impact of reasonable fluctuations (or impulses). While discussing the practical-oriented systems, their corresponding models with non-integer order is more productive and it enhances the accuracy that a system needs. For instance, a proportional-integral-derivative controller (PID controller or three-term controller) is a system with control loop employing feedback that is broadly used in industrial control systems and in wide range of other applications that demand continuous modulated control. Instead of classical controller, Podlubny \cite{Frac-PID} considered $PI^{\lambda}D^{\mu}$ controller combining fractional order integrand ($I^{\lambda}$) and fractional order derivative ($D^{\mu}$). An illustration is also provided in his work that proves that $PI^{\lambda}D^{\mu}$ controller works better than the classical PID controller. While working on fractional model, choosing an effective fractional order is vital.

Generalized Riemann-Liouville fractional derivative, termed later as Hilfer fractional order derivative, arose as a theoretical model of dielectric relaxation in glass-forming materials in a work by Hilfer \cite{Hilfer-glass}. The two classical fractional order derivatives Caputo and Riemann Liouville are a particular case of  Hilfer fractional derivative.
The existence and uniqueness of solution of  systems with Hilfer fractional derivative with different
constraints such as an impulsive system with nonlocal conditions was given by Gou and Li \cite{Hilfer-Impulsive-nonlocal}, approximate controllability of impulsive Hilfer fractional system was given by Jiang and Niazi \cite{Hilfer-impulsive} and with delay conditions was given by Ahmed \textit{et al}.\cite{Hilfer-delay}, etc.

The study of stability analysis for non-integer systems was initially given by Podlubny \textit{et al}.\cite{Mittag-frac}. The work of Stamova \cite{Mittag-Impulsive} on impulsive fractional order is also a classical result.  Using the theory of $q$-calculus, Li \textit{et al}. \cite{LDM} studied the $q$- Mittag-Leffler stability of $q$-fractional differential systems. Ren and Zhai \cite{Stability-Ren} studied the stability conditions for generalized fractional derivative along with examples in neural network. Even though stability analysis has been done for Hilfer fractional system by Rezazadeh \textit{et al}.\cite{Hilfer-stability}, Wang \textit{et al}. \cite{G.Mittag-Hilfer}, stability analysis of impulsive differential system with Hilfer fractional derivative has never been studied. The present work studies the generalized Mittag-Leffler stability of a Hilfer fractional differential systems involving both, instantaneous and non-instantaneous impulsive conditions using Lyapunov approach.

The structure of the paper is in the following sequence. Section 2 covers the essential notions that are used in the rest of the paper. In Section 3, mild solution in integral form of  impulsive differential equations with Hilfer fractional derivative with changeable initial conditions are discussed. In Section 4,  both non-instantaneous and instantaneous impulsive systems are outlined and the lemma which is necessary for the stability analysis is proved. Section 5 elaborates the stability analysis for both the impulsive cases separately.

\section{Essential notions}
Let $t_{0}\in \mathbb{R}_{+}=[0,\infty)$ be the initial time. The fractional integral of order $\mu$ is given as \cite[Sec 2.3.2]{Podlubny-book},
\begin{align*}
_{t_{0}}I^{\mu}_{t}g(t)=\dfrac{1}{\Gamma(\mu)}\int^{t}_{t_{0}}(t-s)^{\mu-1}g(s)ds,\enspace t\geq t_{0},\enspace \enspace 0< \mu <1.
\end{align*}
 Here $\Gamma(\cdot)$ is the well known gamma function and $g$ an integrable function. The two classical derivatives Caputo and
 Riemann-Liouville fractional derivative of order $\mu$ are given by \cite[Sec 2.4.1]{Podlubny-book},
 \begin{align*}
^{C}_{t_{0}}D^{\mu}_{t}x(t)=\dfrac{1}{\Gamma(1-\mu)}\int^{t}_{t_{0}}\dfrac{x'(s)}{(t-s)^{\mu}}ds,\enspace t\geq t_{0},\enspace 0< \mu <1,
\end{align*}
and
\begin{align*}
^{RL}_{t_{0}}D^{\mu}_{t}x(t)=\dfrac{1}{\Gamma(1-\mu)}\left(\dfrac{d}{dt}\right)\int^{t}_{t_{0}}\dfrac{x(s)}{(t-s)^{\mu}}ds,\enspace t\geq t_{0},\enspace 0< \mu <1.
\end{align*}
The Hilfer fractional derivative of order $0< \mu <1$ and type $0\leq \nu \leq 1$, of function $x(t)$ is defined by Hilfer \cite{Hilfer-glass} as
\begin{align*}
(_{t_{0}}D^{\mu,\nu}_{t})x(t)=\big({}_{t_{0}}I_{t}^{\nu(1-\mu)}D(_{t_{0}}I_{t}^{(1-\nu)(1-\mu)})\Big)x(t)
\end{align*}
where $D:=\dfrac{d}{dt}$. For the results regarding the existence of solution of systems with Hilfer fractional derivative, the reader can refer the work of Furati \textit{et al}.\cite{Hilfer-exist-1} and Gu and Trujillo \cite{Hilfer-remark}. Riemann-Liouville and Caputo can be considered as a particular case of Hilfer fractional derivative, respectively as
\begin{align*}
_{t_{0}}D_{t}^{\mu,\nu}=
\left\{
  \begin{array}{ll}
   \nu=0\Rightarrow D\,_{t_{0}}I_{t}^{1-\mu}={ }^{RL}_{t_{0}}D_{t}^{\mu} \\
  \nu=1\Rightarrow {}_{t_{0}}I_{t}^{1-\mu}D={ }^{C}_{t_{0}}D^{\mu}_{t}
  \end{array}
\right.
\end{align*}
The parameter $\lambda$ satisfies $\lambda=\mu+\nu-\mu \nu, \enspace 0<\lambda\leqq 1$.
Another important tool that is used is the Laplace transform of fractional order. For the operator $\mathcal{L}$, the Laplace transform is given by
\begin{align*}
\mathcal{L}\{x(t):s\}:=\int_{0}^{\infty}e^{-st}x(t)dt=:X(s), \enspace \Re(s)>0.
\end{align*}
Here, the function $x$ is assumed to be locally integrable on $[0,\infty)$. The Laplace transform with respect to the Hilfer fractional derivative, was given by  Rezazadeh \textit{et al}. in \cite{Hilfer-stability}, as
\begin{align*}
\mathcal{L}[{}_{0}D_{t}^{\mu,\nu}x(t):s]=s^{\mu}\mathcal{L}[x(t)]-s^{\nu(\mu-1)}({}_{0}I_{t}^{(1-\nu)(1-\mu)}x)(0^{+}), \enspace \Re(s)>0.
\end{align*}

The Mittag-Leffler function with one parameter say $\mu$  and two parameters $\mu$ and $\lambda$ are given respectively as, (see \cite{Podlubny-book} for details)
\begin{align*}
  E_{\mu}(z)=\sum_{k=0}^{\infty}\dfrac{z^{k}}{\Gamma(\mu k+1)} \enspace \mbox{and}\enspace E_{\mu,\lambda}(z)=\sum_{k=0}^{\infty}\dfrac{z^{k}}{\Gamma(\mu k+\lambda)},\enspace \mu>0,\enspace\lambda>0, \enspace z \in \mathfrak{C}.
\end{align*}
and for $\lambda=1$, $E_{\mu,1}(z)=E_{\mu}(z)$.

The Laplace transforms with respect to one and two parameter Mittag-Leffler function are given respectively as
\begin{align*}
\mathcal{L}\{E_{\mu}(-\gamma t^{\mu})\}= \dfrac{s^{\mu-1}}{s^{\mu}+\gamma} \enspace \mbox{and} \enspace
\mathcal{L}\{t^{\lambda-1}E_{\mu,\lambda}(-\gamma t^{\mu})\}= \dfrac{s^{\mu-\lambda}}{s^{\mu}+\gamma}, \enspace \gamma\in \mathbb{R}.
\end{align*}

\begin{proposition}\cite{Hilfer-exist-1}
For $\mu \in (0,1)$, $\delta>0$ the following statements are true.
\begin{align*}
{}_{t_{0}}I^{\mu}_{t}(t-t_{0})^{\delta-1}=\dfrac{\Gamma(\delta)}{\Gamma(\delta+\mu)}(t-t_{0})^{\delta+\mu-1}\\
{}_{t_{0}}D^{\mu,\nu}_{t}(t-t_{0})^{\delta-1}=\dfrac{\Gamma(\delta)}{\Gamma(\delta-\mu)}(t-t_{0})^{\delta-\mu-1}
\end{align*}
\end{proposition}
\begin{corollary}\cite{Hilfer-exist-1}
Further, for $\lambda \in(0,1)$, $\mu \in (0,1)$, the following statements can be derived from the above proposition.
\begin{flalign*}
&{}_{t_{0}}I_{t}^{\mu}(t-t_{0})^{-\mu}=\Gamma(1-\mu)\\
&{}_{t_{0}}D^{\mu,\nu}_{t}1=\dfrac{1}{\Gamma(1-\mu)}(t-t_{0})^{-\mu}\\
&{}_{t_{0}}D^{\mu,\nu}_{t}(t-t_{0})^{\lambda-1}=0
\end{flalign*}
\end{corollary}
Consider the nonlinear Hilfer fractional differential equation
\begin{align}
\label{eqn:stability-general}
{}_{a}D^{\mu,\nu}_{t}x(t)=G(t,x(t)),\enspace t>a,\enspace 0<\mu<1,\enspace 0\leq \nu \leq 1.
\end{align}
The two commonly used initial conditions of \eqref{eqn:stability-general} are as follows:
\begin{enumerate}
\item
Initial condition in integral form:
\begin{align}
\label{stability-integral}
{}_{a}I^{1-\lambda}_{t}x(t)|_{t=a}=B\in \mathbb{R}.
\end{align}
\item
Initial condition in weighted Cauchy type problem:
\begin{align}
\label{stability-weighted}
\lim_{t\rightarrow a}\left((t-a)^{1-\lambda}x(t)\right)=C\in \mathbb{R}.
\end{align}
\end{enumerate}
\begin{remark}
\label{rem:Initial}
In accordance with \cite[Lemma 3.2]{Kilbas-book}, if any function $x(t)$ satisfies the initial condition \eqref{stability-integral}, then $x(t)$ also satisfies the condition \eqref{stability-weighted} with $B=C\Gamma(\lambda)$. Also \eqref{eqn:stability-general} with the above two forms of initial conditions will have an equivalent integral representations.
\end{remark}
 In this paper, the initial condition in weighted form \eqref{stability-weighted} is taken into consideration and the following lemma gives the integral representation for the weighted form of initial condition.
 Let $C[a,b]$ and $C^{n}[a,b]$ be the space of continuous functions and n times continuously differentiable functions on $[a,b]$, respectively. In general, the weighted space of continuous functions are given by
 \begin{align*}
 C_{1-\lambda}[a,b]&=\{g:(a,b]\rightarrow \mathbb{R}:(x-a)^{1-\lambda}g(x)\in C[a,b]\},\enspace 0< \lambda \leq 1,\\
  C_{1-\lambda}^{\lambda}[a,b]&=\{g\in C_{1-\lambda}[a,b],\enspace D^{\lambda}_{a^{+}}g\in C_{1-\lambda}[a,b]\}.
 \end{align*}
 \begin{lemma}\cite{Hilfer-exist-1}
 Let $\lambda=\mu+\nu-\mu\nu$ where $0<\mu<1$ and $0\leq \nu \leq 1$. Let $G:(a,b]\times \mathbb{R}\rightarrow \mathbb{R}$ be a function such that $G(\cdot,x(\cdot))\in C_{1-\lambda}[a,b]$ for any $x\in C^{\lambda}_{1-\lambda}[a,b]$, then $x(t)$ satisfies \eqref{eqn:stability-general}-\eqref{stability-weighted} if and only if $x$ satisfies
 \begin{align*}
 x(t)=C(t-a)^{\lambda-1}+\dfrac{1}{\Gamma(\mu)}\int^{t}_{a}\dfrac{G(s,x(s))}{(t-s)^{1-\mu}}ds,\enspace t\in (a,b].
 \end{align*}
 \end{lemma}
The following lemma is on the uniqueness of the solution of the Cauchy type problem \eqref{eqn:stability-general}-\eqref{stability-weighted} in the space $C^{\lambda}_{1-\lambda}[a,b]$. Also, the global existence for the weighted Cauchy type problem for Hilfer fractional differential equation is proved by Furati \cite{Hilfer-exist-1}.
\begin{lemma}\cite{Hilfer-exist-1}
Let $0<\mu<1$, $0\leq \nu \leq 1$ and $\lambda=\mu+\nu-\mu\nu$. Let $G:(a,b]\times \mathbb{R}\rightarrow \mathbb{R}$ be a function such that $G(\cdot,x(\cdot))\in C^{\lambda}_{1-\lambda}[a,b]$ for any $x\in C_{1-\lambda}[a,b]$ and satisfy the Lipschitz condition such that $L>0$ exists with $|G(t,x)-G(t,y)|\leq L|x-y|$ . Then there exists an unique solution to the initial value problem \eqref{eqn:stability-general}-\eqref{stability-weighted} in the space $C^{\lambda}_{1-\lambda}[a,b]$.
\end{lemma}
\begin{lemma} \cite{Hilfer-exist-1}
\label{lem:4-stability}
\begin{enumerate}[\rm(1)]
\item
If $x\in C(t_{0},T]$, then for any point $t\in (t_{0},T]$
\begin{align*}
{}_{t_{0}}D^{\mu,\nu}_{t}({}_{t_{0}}I^{\mu}_{t}x(t))=x(t).
\end{align*}
\item
If $x\in C(t_{0},T]$ and ${}_{t_{0}}I^{\mu}_{t}x(t)\in C(t_{0},T]$, then for any point $t\in (t_{0},T]$,
\begin{align*}
{}_{t_{0}}I^{\mu}_{t}({}_{t_{0}}D^{\mu,\nu}_{t}x(t))=x(t)-\dfrac{(t-t_{0})^{\lambda-1}}{\Gamma(\lambda)}{}_{t_{0}}I^{(1-\mu)(1-\mu)}_{t}x(t)|_{t=t_{0}}.
\end{align*}
\end{enumerate}
\end{lemma}
\section{Existence of mild solution with changed lower bounds of the Hilfer fractional derivative at the impulsive points}
 \subsection{The case when the Impulses are Non-instantaneous}
The initial value problem (IVP) with Hilfer fractional differential equations with non-instantaneous impulses is given by
\begin{align}
\label{stability-NIHrFrDE}
{}_{t_{i}}D^{\mu,\nu}_{t}x(t)&=g(t,x), \enspace t\in(t_{i},p_{i}],\enspace i=0,1,\ldots k,
\end{align}
with weighted impulsive and initial condition,
\begin{eqnarray}
\label{stability-NIImpulse}
\left\{
  \begin{array}{ll}
  x(t)=\phi_{i}(t,x(t),x(p_{i}-0)),\enspace t\in (p_{i},t_{i+1}],\enspace i=0,1,2,\ldots k-1,\\
\lim_{t\rightarrow t_{i}}\big((t-t_{i})^{1-\lambda}x(t)\big)=\phi_{i-1}(t_{i},x(t_{i}),x(p_{i-1}-0)),\enspace i=1,2,\ldots k-1,\\
\lim_{t\rightarrow t_{0}}\big((t-t_{0})^{1-\lambda}x(t)\big)=x_{0}.
  \end{array}
\right.
\end{eqnarray}
Here $x_{0}\in \mathbb{R}^{n}$, $g:\cup_{i=0}^{k}[t_{i},p_{i}]\times \mathbb{R}^{n}\rightarrow \mathbb{R}^{n}$, $\phi_{i}:[p_{i},t_{i+1}]\times \mathbb{R}^{n}\times \mathbb{R}^{n}\rightarrow \mathbb{R}^{n}$ for ($i=0,1,2,\ldots k$).
\begin{definition}
A function $x:[t_{0},T]\rightarrow \mathbb{R}^{n}$ is called a mild solution of the IVP for non-instantaneous Hilfer fractional differential system \eqref{stability-NIHrFrDE}, if it satisfies the following Volterra-algebraic equation
\begin{eqnarray*}\
%\label{eqn:Stability-mild-NI}
x(t)=
\left\{
  \begin{array}{ll}
  x_{0}(t-t_{0})^{\lambda-1}+\dfrac{1}{\Gamma(\mu)}\displaystyle\int_{t_{0}}^{t}(t-s)^{\mu-1}g(s,x(s))ds,\enspace t\in(t_{0},p_{0}],\\
  \phi_{i}(t,x(t),x(p_{i}-0)),\enspace t\in(p_{i},t_{i+1}],\enspace i=0,1,2,\ldots k-1,\\
   \phi_{i-1}(t_{i},x(t_{i}),x(p_{i-1}-0))(t-t_{i})^{\lambda-1}\\
   +\dfrac{1}{\Gamma(\mu)}\displaystyle\int_{t_{i}}^{t}(t-s)^{\mu-1}g(s,x(s))ds, \enspace t\in(t_{i},p_{i}],\enspace i=1,2,\ldots k.
  \end{array}
\right.
\end{eqnarray*}
\end{definition}
The following theorem gives the condition for the existence and uniqueness of a mild solution for the fractional system \eqref{stability-NIHrFrDE} with impulsive and initial conditions given in \eqref{stability-NIImpulse}.
\begin{theorem}
\label{thm:Stability-NIHr}
The weighted form of IVP \eqref{stability-NIHrFrDE} has a unique mild solution if the following assumptions are satisfied.
\begin{enumerate}[\rm(1)]
\item
For $x_{1},x_{2} \in \mathbb{R}^{n}$, the function $g:\cup_{i=0}^{k}[t_{i},p_{i}]\times \mathbb{R}^{n}\rightarrow \mathbb{R}^{n}$, $g(t,x)\in \cup_{i=0}^{k}C_{1-\lambda}[t_{i},p_{i}]$ satisfies the inequality,
\begin{align*}
\|g(t,x_{1})-g(t,x_{2})\|\leq L\|x_{1}-x_{2}\|, \enspace \mbox{for}\enspace  L>0 \enspace \forall t\in[t_{i},p_{i}].
\end{align*}
 \item
For $x_{1},x_{2},y_{1},y_{2} \in \mathbb{R}^{n}$, the function $\phi_{i}(t,x,y)\in C[p_{i},t_{i+1}]$  satisfies the inequality,
\begin{align*}
\|\phi_{i}(t,x_{1},y_{1})-\phi_{i}(t,x_{2},y_{2})\|\leq I_{i}\big(\|x_{1}-x_{2}\|+\|y_{1}-y_{2}\|\big),
\end{align*}
for $t\in [p_{i},t_{i+1}]$, and $ I_{i}>0, \enspace (i=0,1,\ldots k-1)$.
\item
The inequality $K<1$  holds where,
\small{
\begin{align*}
K=\max \Bigg(\max_{i=0,1,\ldots k}I_{i},\enspace& \dfrac{L(t-t_{0})^{\mu}}{\Gamma(\mu)}\Big(\frac{1-p}{\lambda-p}\Big)^{1-p}\Big(\dfrac{p}{p+\mu-1}\Big)^{p},\enspace I_{i}+\dfrac{I_{i}}{(p_{i-1}-t_{i-1})^{1-\lambda}}\\
&+\dfrac{L(t-t_{i})^{\mu}}{\Gamma(\mu)}\Big(\frac{1-p}{\lambda-p}\Big)^{1-p}\Big(\dfrac{p}{p+\mu-1}\Big)^{p}\Bigg).
\end{align*}
}
\end{enumerate}
\end{theorem}
\begin{proof}
The theorem is proved using Banach contraction principle. An operator $\mathcal{G}$ is defined for any function $x\in PC_{1-\lambda}[t_{0},T]$ as
\begin{eqnarray}
\label{eqn:stability-operator}
\mathcal{G}x(t)=
\left\{
  \begin{array}{ll}
  x_{0}(t-t_{0})^{\lambda-1}+\dfrac{1}{\Gamma(\mu)}\displaystyle\int_{t_{0}}^{t}(t-s)^{\mu-1}g(s,x(s))ds,\enspace t\in(t_{0},p_{0}],\\
    \phi_{i}(t,x(t),x(p_{i}-0)),\enspace t\in(p_{i},t_{i+1}],\enspace i=0,1,2,\ldots k-1,\\
   \phi_{i-1}(t_{i},x(t_{i}),x(p_{i-1}-0))(t-t_{i})^{\lambda-1}\\
   +\dfrac{1}{\Gamma(\mu)}\displaystyle\int_{t_{i}}^{t}(t-s)^{\mu-1}g(s,x(s))ds, \enspace t\in(t_{i},p_{i}],\enspace i=1,2,\ldots k.
  \end{array}
\right.
\end{eqnarray}
From the assumption $1$,  given in the theorem statement, it can observed that the operator $\mathcal{G}x(t)$ is well defined. The theorem can be proved by a series of steps.

\underline{Step 1}. To prove that  $\mathcal{G}x(t)\in PC_{1-\lambda}[t_{0},T]$ for $x\in PC_{1-\lambda}[t_{0},T]$.\\
From the definition of the operator $\mathcal{G}$ in \eqref{eqn:stability-operator}, it is obvious that $$\mathcal{G}x(t)\in \cup^{k}_{i=0}C(t_{i},p_{i})\bigcup \cup_{i=0}^{k}C(p_{i},t_{i+1}).$$
For the case $t\in (t_{0},p_{0}]$,
\begin{align*}
(t-t_{0})^{1-\lambda}\mathcal{G}x(t)&=x_{0}+\dfrac{1}{\Gamma(\mu)}\displaystyle\int_{t_{0}}^{t}\dfrac{(t-t_{0})^{1-\lambda}}{(t-s)^{1-\mu}}g(s,x(s))ds.\\
\Longrightarrow &(t-t_{0})^{1-\lambda}\mathcal{G}x(t) \in C(t_{0},p_{0}].
\end{align*}
For the case $t\in (t_{i},p_{i}]$, $i=1,2,\ldots k$,
\begin{align*}
(t-t_{i})^{1-\lambda}\mathcal{G}x(t)&=\phi_{i-1}(t_{i},x_{i}(t),x(p_{i-1}-0))+\dfrac{1}{\Gamma(\mu)}\displaystyle\int_{t_{i}}^{t}\dfrac{(t-t_{i})^{1-\lambda}}{(t-s)^{1-\mu}}g(s,x(s))ds.\\
\Longrightarrow (t-t_{i})^{1-\lambda}\mathcal{G}x(t) &\in C(t_{i},p_{i}].
\end{align*}
Hence proved.

\underline{Step 2.} The proof will be complete if it can be proved that $\mathcal{G}$ is a contraction operator in $PC_{1-\lambda}[t_{0},T]$. Let $x_{1},x_{2}\in  PC_{1-\lambda}[t_{0},T]$. \\
Consider the case $t\in (t_{0},p_{0}]$,
\begin{align*}
\sup_{t\in[t_{0},p_{0}]}&\|(t-t_{0})^{1-\lambda}\mathcal{G}x_{1}(t)-(t-t_{0})^{1-\lambda}\mathcal{G}x_{2}(t)\|\\
&\leq \dfrac{L(t-t_{0})^{1-\lambda}}{\Gamma(\mu)}\int_{t_{0}}^{t}(t-s)^{\mu-1}\|x_{1}(s)-x_{2}(s)\|ds\\
&=\dfrac{L(t-t_{0})^{1-\lambda}}{\Gamma(\mu)}\int_{t_{0}}^{t}\dfrac{(t-s)^{\mu-1}}{(s-t_{0})^{1-\lambda}}
\|(s-t_{0})^{1-\lambda}x_{1}(s)-(s-t_{0})^{1-\lambda}x_{2}(s)\|ds\\
&\leq \|x_{1}-x_{2}\|_{PC_{1-\lambda}[t_{0},T]}\dfrac{L(t-t_{0})^{1-\lambda}}{\Gamma(\mu)}\int_{t_{0}}^{t}\dfrac{(t-s)^{\mu-1}}{(s-t_{0})^{1-\lambda}}ds.
\end{align*}
To proceed further, H$\ddot{o}$lder's inequality is applied and it leads to
\small{
\begin{align*}
&\sup_{t\in[t_{0},p_{0}]}\|(t-t_{0})^{1-\lambda}\mathcal{G}x_{1}(t)-(t-t_{0})^{1-\lambda}\mathcal{G}x_{2}(t)\|\\
&\leq \|x_{1}-x_{2}\|_{PC_{1-\lambda}[t_{0},T]}\dfrac{L(t-t_{0})^{1-\lambda}}{\Gamma(\mu)}\Bigg[\Big(\int_{t_{0}}^{t}(s-t_{0})^{\frac{\lambda-1}{1-p}}ds\Big)^{1-p}
\Big(\int_{t_{0}}^{t}(t-s)^{\frac{\mu-1}{p}}ds\Big)^{p}\Bigg]\\
&= \|x_{1}-x_{2}\|_{PC_{1-\lambda}[t_{0},T]}\dfrac{L(t-t_{0})^{1-\lambda}}{\Gamma(\mu)}\Bigg[\Big(\dfrac{1-p}{\lambda-p}\Big)^{1-p}(t-t_{0})^{\lambda-p}
\Big(\dfrac{p}{p+\mu-1}\Big)^{p}(t-t_{0})^{p+\mu-1}\Bigg]\\
&= \|x_{1}-x_{2}\|_{PC_{1-\lambda}[t_{0},T]}\dfrac{L(t-t_{0})^{\mu}}{\Gamma(\mu)}\Big(\frac{1-p}{\lambda-p}\Big)^{1-p}\Big(\dfrac{p}{p+\mu-1}\Big)^{p}.
\end{align*}
}
\begin{align}
\label{eqn:Stability-K1}
\Longrightarrow\sup_{t\in[t_{0},p_{0}]}\|(t-t_{0})^{1-\lambda}\mathcal{G}x_{1}(t)-(t-t_{0})^{1-\lambda}\mathcal{G}x_{2}(t)\|\leq K\|x_{1}-x_{2}\|_{PC_{1-\lambda}[t_{0},T]}.
\end{align}
For the case $t\in (p_{0},t_{1}]$,
\begin{align*}
\sup_{(p_{0},t_{1}]}\|\mathcal{G}x_{1}(t)-\mathcal{G}x_{2}(t)\|=&\sup_{(p_{0},t_{1}]}\|\phi_{0}\big(t,x_{1}(p_{0}-0)\big)-\phi_{0}\big(t,x_{2}(p_{0}-0)\big)\|\\
\leq &\enspace I_{0}\|x_{1}-x_{2}\|_{PC_{1-\lambda}[t_{0},T]}.
\end{align*}
In a similar way for $t\in (p_{i},t_{i+1}]$,
\begin{align}
\label{eqn:Stability-K2}
\sup_{(p_{i},t_{i+1}]}\|\mathcal{G}x_{1}(t)-\mathcal{G}x_{2}(t)\|=&\sup_{(p_{i},t_{i+1}]}\|\phi_{i}\big(t,x_{1}(p_{i}-0)\big)-\phi_{i}\big(t,x_{2}(p_{i}-0)\big)\|\nonumber\\
\leq & I_{i}\|x_{1}-x_{2}\|_{PC_{1-\lambda}[t_{0},T]}\leq K\|x_{1}-x_{2}\|_{PC_{1-\lambda}[t_{0},T]}.
\end{align}
For the general case $t\in(t_{i},p_{i}]$, the calculation can be proceeded as below:
\begin{align*}
\sup_{t\in (t_{i},p_{i}]}&\|(t-t_{i})^{1-\lambda}\mathcal{G}x_{1}(t)-(t-t_{i})^{1-\lambda}\mathcal{G}x_{2}(t)\|\\
\leq& \|\phi_{i-1}\big(t_{i},x_{1}(t_{i}),x_{1}(p_{i-1}-0)\big)-\phi_{i-1}\big(t_{i},x_{2}(t_{i}),x_{2}(p_{i-1}-0)\big)\|\\
+&\dfrac{1}{\Gamma(\mu)}\int_{t_{i}}^{t}(t-t_{i})^{1-\lambda}(t-s)^{\mu-1}\|g\big(s,x_{1}(s)\big)-g\big(s,x_{2}(s)\big)\|ds\\
\leq& I_{i}\|x_{1}(t_{i})-x_{2}(t_{i})\|\\
+&\dfrac{I_{i}}{(p_{i-1}-t_{i-1})^{1-\lambda}}(p_{i-1}-t_{i-1})^{1-\lambda}\big\|x_{1}(p_{i-1}-0)-x_{2}(p_{i-1}-0)\big\|\\
+& \dfrac{L(t-t_{i})^{1-\lambda}}{\Gamma(\mu)}\int_{t_{i}}^{t}\dfrac{(t-s)^{\mu-1}}{(s-t_{i})^{1-\lambda}}
\big\|(s-t_{i})^{1-\lambda}x_{1}(s)-(s-t_{i})^{1-\lambda}x_{2}(s)\big\|ds.
\end{align*}
From the fact that
\begin{align*}
\|x_{1}(t_{i})-x_{2}(t_{i})\|\leq \sup_{t\in (p_{i},t_{i+1}]}\|x_{1}(t)-x_{2}(t)\|\leq \|x_{1}-x_{2}\|_{PC_{1-\lambda}[t_{0},T]},
\end{align*}
and
\begin{align*}
\big\|(p_{i-1}-t_{i-1})^{1-\lambda}&\big(x_{1}(p_{i-1}-0)-x_{2}(p_{i-1}-0)\big)\big\|\\
&\leq \sup_{t\in (t_{i-1},p_{i-1}]}\big\|(t-t_{i})^{1-\lambda}\big(x_{1}(t)-x_{2}(t)\big)\big\|\leq \|x_{1}-x_{2}\|_{PC_{1-\lambda}[t_{0},T]},
\end{align*}
the following conclusion can be drawn. That is,
\begin{align}
\label{eqn:Stability-K3}
&\sup_{t\in (t_{i},p_{i}]}\|(t-t_{i})^{1-\lambda}\mathcal{G}x_{1}(t)-(t-t_{i})^{1-\lambda}\mathcal{G}x_{2}(t)\|\nonumber\\
\leq&\Bigg[I_{i}+\dfrac{I_{i}}{(p_{i-1}-t_{i-1})^{1-\lambda}}
+\dfrac{L(t-t_{i})^{\mu}}{\Gamma(\mu)}\Big(\frac{1-p}{\lambda-p}\Big)^{1-p}\Big(\dfrac{p}{p+\mu-1}\Big)^{p}\Bigg]\big\|x_{1}-x_{2}\big\|_{PC_{1-\lambda}[t_{0},T]}\nonumber\\
\leq & K\|x_{1}-x_{2}\|_{PC_{1-\lambda}[t_{0},T]}.
\end{align}
From the inequalities \eqref{eqn:Stability-K1}, \eqref{eqn:Stability-K2}, \eqref{eqn:Stability-K3}  and from the assumption (3) of the given theorem hypothesis, the theorem is proved.
\end{proof}
\subsection{The case when the Impulses are Instantaneous}
Consider the IVP of Hilfer fractional differential equations with instantaneous impulses as below:
\begin{align}
\label{stability-IHrFrDE}
{}_{t_{i}}D^{\mu,\nu}_{t}x(t)&=g(t,x), \enspace t\in(t_{i},t_{i+1}],\enspace i=0,1,\ldots k,
\end{align}
with weighted impulsive and initial condition,
\begin{eqnarray*}
%\label{stability-IImpulse}
\left\{
  \begin{array}{ll}
  x(t)=\psi_{i}\big(t,x(t_{i}-0)\big),\enspace t=t_{i},\enspace i=1,2,\ldots k,\\
\lim_{t\rightarrow t_{i}}\big((t-t_{i})^{1-\lambda}x(t)\big)=\psi_{i}(t_{i},x(t_{i}-0)),\enspace i=1,2,\ldots k,\\
\lim_{t\rightarrow t_{0}}\big((t-t_{0})^{1-\lambda}x(t)\big)=x_{0}.
  \end{array}
\right.
\end{eqnarray*}
Here $x_{0}\in \mathbb{R}^{n}$, $g:\cup_{i=0}^{k}[t_{i},t_{i+1}]\times \mathbb{R}^{n}\rightarrow \mathbb{R}^{n}$ and $\psi_{i}: [t_{o},T]\times \mathbb{R}^{n}\rightarrow \mathbb{R}^{n}$ for ($i=0,1,2,\ldots k$).
\begin{definition}
A function $x:[t_{0},T]\rightarrow \mathbb{R}^{n}$ is called a mild solution of the IVP for instantaneous Hilfer fractional differential system \eqref{stability-IHrFrDE} if it satisfies the following Volterra-algebraic equation
\begin{eqnarray*}
x(t)=
\left\{
  \begin{array}{ll}
   \psi_{i}(t,x(t_{i}-0)),\enspace t=t_{i},\enspace i=1,2,\ldots k,\\
   \psi_{i}(t_{i},x(t_{i}-0))(t-t_{i})^{\lambda-1}\\
   +\dfrac{1}{\Gamma(\mu)}\displaystyle\int_{t_{i}}^{t}(t-s)^{\mu-1}g(s,x(s))ds, \enspace t\in(t_{i},t_{i+1}],\enspace i=1,2,\ldots k,\\
   \psi_{0}(t,x(t_{0}-0))=x_{0}.
  \end{array}
\right.
\end{eqnarray*}
\end{definition}
\begin{remark}
The existence and uniqueness of the mild solution of the system \eqref{stability-IHrFrDE} can be proved in a similar way as Theorem \ref{thm:Stability-NIHr}.
\end{remark}
\section{Stability Analysis}
While analyzing the nonlinear systems, stability is vital, amongst other significant traits in control theory. In 1996, the stability of
linear fractional systems were studied by Matignon \cite{stability-frac} concerning the Caputo derivative. Subsequently, a substantial study has been done by many authors on the stability theory of fractional systems. The work of Li \textit{et al.} \cite{Mittag-frac}  on stability impelled many authors to peruse the study on Mittag-Leffler stability and Lyapunov direct method. The following is the definition of the generalized Mittag-Leffler stability of the solution $x(t)$ of the Hilfer fractional differential system.
\begin{definition}
Let $\Omega\subset\mathbb{R}$ be a domain containing the origin. The solution  is said to be generalized Mittag-Leffler Stable if
\begin{align*}
%\label{eqn:Mittag-condition}
\|x(t)\|\leq\big[m[{}_{t_{0}}I_{t}^{1-\lambda}x(t_{0})](t-t_{0})^{\lambda-1}E_{\mu,\lambda}\big(-\gamma(t-t_{0})^{\mu}\big)\big]^c,
\end{align*}
where $\mu\in (0,1)$, $\nu\in[0,1]$, $\lambda=\nu+\nu-\mu\nu$, $\gamma\geq 0$, $m(0)=0$, $m(x)\geq 0$ and $m(x)$ is locally Lipschitz with Lipschitz constant $m_{0}$ and ${}_{t_{0}}I_{t}^{1-\lambda}x(t_{0})$ is the integral type initial condition.
\end{definition}
We discuss the stability analysis of the two impulsive system in the subsequent subsections separately.
\label{Stability-stability}
\subsection{For Non-instantaneous Impulsive System}
 For the non-instantaneous impulsive system with Hilfer fractional order, consider the sequences $\{t_{i}\}_{i=1}^{\infty}$, $\{p_{i}\}^{\infty}_{i=0}$, with $0\leq t_{0}<p_{0}<t_{i}<p_{i}<t_{i+1}\leq p_{i+1}$ for $i=1,2,\ldots $ and $\displaystyle\lim_{i\rightarrow \infty}t_{i}=\infty$. The IVP with non-instantaneous impulses and Hilfer order derivative is given by
 \begin{align}
\label{stability-NIHrFrDE-1}
{}_{t_{i}}D^{\mu,\nu}_{t}x(t)&=g(t,x), \enspace t\in(t_{i},p_{i}],\enspace i=0,1,\ldots ,
\end{align}
with weighted impulsive and initial condition,
\begin{eqnarray*}
%\label{stability-NIImpulse-1}
\left\{
  \begin{array}{ll}
  x(t)=\phi_{i}(t,x(t),x(p_{i}-0)),\enspace t\in (p_{i},t_{i+1}],\enspace i=0,1,2,\ldots ,\\
\lim_{t\rightarrow t_{i}}\big((t-t_{i})^{1-\lambda}x(t)\big)=\phi_{i}(t_{i},x(t_{i}),x(p_{i}-0)),\enspace i=0,1,2,\ldots, \\
\lim_{t\rightarrow t_{0}}\big((t-t_{0})^{1-\lambda}x(t)\big)=x_{0}.
  \end{array}
\right.
\end{eqnarray*}
Here $x_{0}\in \mathbb{R}^{n}$, $g:\cup_{i=0}^{\infty}[t_{i},p_{i}]\times \mathbb{R}^{n}\rightarrow \mathbb{R}^{n}$, $\phi_{i}:[p_{i},t_{i+1}]\times \mathbb{R}^{n}\times \mathbb{R}^{n}\rightarrow \mathbb{R}^{n}$ for ($i=0,1,2,\ldots $). \\
For an arbitrary initial value $\tau\in [t_{i},p_{i})$, $i=0,1,\ldots$, a general Hilfer fractional IVP can be described as
\begin{equation}
\begin{split}
\label{eqn:Hilfer-mittag-imp-gen}
{}_{\tau}D_{t}^{\mu,\nu}x(t)= &g\big(t,x(t)\big), \enspace t\in [\tau,p_{i}],\\
\lim_{t\rightarrow \tau}\big((t-\tau)^{1-\lambda}x(t)\big)=& \widehat{x}_{0}.
\end{split}
\end{equation}
The solution of the IVP of Hilfer fractional non-instantaneous differential system  \eqref{stability-NIHrFrDE-1}, with variable initial condition is given by
\begin{align*}
x(t)=x(t;t_{0},x_{0}) =&\left\{
  \begin{array}{lll}
   X_{i}(t), & t\in (t_{i},p_{i}],\enspace & i=0,1,\ldots\nonumber\\
  \phi_{i}(t,x(t),X_{i}(p_{i}-0)), & t\in (p_{i},t_{i+1}],\enspace & i=0,1,2,\ldots
 \end{array}
\right.
\end{align*}
The following remark is the output of the interval-by-interval analysis of the IVP with the Hilfer fractional derivative given above.
\begin{remark}
\label{rem:Hilfer-mittag-solu}
\begin{enumerate}[\rm(1)]
\item
For $t\in[t_{0},p_{0}]$, the solution $X_{0}(t)$ of the system \eqref{stability-NIHrFrDE-1} coincides with the solution of IVP \eqref{eqn:Hilfer-mittag-imp-gen} for $\tau=t_{0}$, $i=0$ and $\widehat{x}_{0}=x_{0};$
\item
For $t\in(p_{0},t_{1}]$, the solution of the system \eqref{stability-NIHrFrDE-1} satisfies the system
\begin{align*}
x(t;t_{0},x_{0})= \phi_{0}\big(t,x(t;t_{0},x_{0}),X_{1}(p_{0}-0)\big);
\end{align*}
\item
For $t\in(t_{1},p_{1}]$, the solution $X_{1}(t)$ of the system \eqref{stability-NIHrFrDE-1} coincides with the solution of IVP \eqref{eqn:Hilfer-mittag-imp-gen} for $\tau=t_{1}$, $i=1$ and $\widehat{x}_{0}=\phi_{0}\big(t_{1},x(t_{1};t_{0},x_{0}),X_{1}(p_{0}-0)\big);$
\item
For $t\in (p_{1},t_{2}]$, the solution of the system \eqref{stability-NIHrFrDE-1} satisfies the system
\begin{align*}
x(t;t_{0},x_{0})= \phi_{1}\big(t,x(t;t_{0},x_{0}),X_{2}(p_{1}-0)\big);
\end{align*}
\item
For $t\in (t_{2},p_{2}]$, the solution $X_{2}(t)$ of the system \eqref{stability-NIHrFrDE-1} coincides with the solution of IVP \eqref{eqn:Hilfer-mittag-imp-gen} for $\tau=t_{2}$, $i=2$ and $\widehat{x}_{0}=\phi_{1}\big(t_{2},x(t_{2};t_{0},x_{0}),X_{2}(p_{1}-0)\big);$
\end{enumerate}
and so on.
\end{remark}
In general, the solution $x(t)$, $t\geq t_{0}$ satisfies the integral system
\begin{align}
\label{eqn:Hilfer-mittag-non-inst-soln}
x(t) =\left\{
  \begin{array}{ll}
x_{0}(t-t_{0})^{\lambda-1}+\dfrac{1}{\Gamma(\mu)}\displaystyle\int_{t_{0}}^{t}(t-s)^{\mu-1}g(s,x(s))ds,\enspace t\in[t_{0},p_{0}],\\
  \phi_{i}(t,x(t),x(p_{i}-0)),\enspace t\in(p_{i},t_{i+1}],\enspace i=0,1,2,\ldots ,\\
   \phi_{i-1}(t_{i},x(t_{i}),x(p_{i-1}-0))(t-t_{i})^{\lambda-1}\\
   +\dfrac{1}{\Gamma(\mu)}\displaystyle\int_{t_{i}}^{t}(t-s)^{\mu-1}g(s,x(s))ds, \enspace t\in(t_{i},p_{i}],\enspace i=1,2,\ldots .
    \end{array}
\right.
\end{align}
The following definition is proposed, for the zero solution of a non-instantaneous system with Hilfer fractional derivative to be Mittag-Leffler stable.
\begin{definition}
\label{def:stability-NHrFDE}
The zero solution of the non-instantaneous impulsive system \eqref{stability-NIHrFrDE-1} with Hilfer fractional derivative of order $0<\mu<1$ and type $0\leq\nu\leq 1$ is said to be  Mittag-Leffler stable if there exist constants $a,b,h,\gamma$ such that for any initial time $t_{0}\in [0,p_{0})\bigcup_{i=1}^{\infty}[t_{i},p_{i})$ the following inequality holds.

\small{
\begin{align*}
\|x(t)\|\leq
\left\{
  \begin{array}{ll}
  h\|x_{0}\|^{b}\Big[\Big(\displaystyle\prod_{l=0}^{i-1}(p_{l}-t_{l})^{\lambda-1}E_{\mu,\nu}\big(-\gamma(p_{l}-t_{l})^{\mu}\big)\Big)(t-t_{i})^{\lambda-1}E_{\mu,\nu}\big(-\gamma(t-t_{i})^{\mu}\big)\Big]^{\frac{1}{a}},\\
  \hspace{2cm} t\in [t_{i},p_{i}], i=1,2,\ldots\\
  h\|x_{0}\|^{b}\Big[\Big(\displaystyle\prod_{l=0}^{i}(p_{l}-t_{l})^{\lambda-1}E_{\mu,\nu}\big(-\gamma(p_{l}-t_{l})^{\mu})\Big)\Big]^{\frac{1}{a}}, \enspace t\in [p_{i},t_{i+1}],\enspace i=0,1,\ldots.
      \end{array}
\right.
\end{align*}
}
\end{definition}
An example is given below to provide a more detailed view of the solution of the non-instantaneous impulsive system with Hilfer fractional derivative and its particular case reducing to Caputo and Riemann-Liouville derivative.
\begin{example}{\rm{\cite{Mittag-Impulsive-inst-noninst}}}.
Consider the IVP with $\mu=0.4$, $g(t,x)= t$, $t_{i}=i$, $p_{i}=0.5+i$, $i=0,1,\ldots$. Let $\phi_{i}(t,x,y)=t-ix+y$ for $(p_{i},t_{i+1}]$, $i=0,1,2\ldots$ with the solution $x(t)=t-ix(t)+x(p_{i}-0)$ or $x(t)=\dfrac{t+x(p_{i}-0)}{1+i}$ in the interval $(p_{i},t_{i+1}]$.  Let
$$G(t,a)=\displaystyle\int_{a}^{t}\dfrac{g(s,x(s))}{(t-s)^{1-\mu}}ds.$$
 For the given example, let $g(s,x(s))=t$. Hence, $G(t,a)=\displaystyle\int_{a}^{t}\dfrac{s}{(t-s)^{0.6}}ds$ and using Mathematica 12.2, the value of $G(t,a)$ is calculated as below:
\begin{align*}
G(t,a)=-2.46265\times10^{-16} t^{1.4} + 0.714286 a (-a + t)^{(2/5)} + 1.78571 t^{1} (-a + t)^{(2/5)}
\end{align*}
The solution derived in \eqref{eqn:Hilfer-mittag-non-inst-soln}, for this example can be calculated as,
\begin{align*}
x(t) =\left\{
  \begin{array}{ll}
  x_{0}t^{\lambda-1}+\dfrac{G(t,0)}{\Gamma(0.4)}, &t\in (0,0.5],\\
  t+x_{0}(0.5)^{\lambda-1}+\dfrac{G(0.5,0)}{\Gamma(0.4)}, &t\in (0.5,1],\\
  \Big(1+x_{0}(0.5)^{\lambda-1}+\dfrac{G(0.5,0)}{\Gamma(0.4)}\Big)(t-1)^{\lambda-1}+\dfrac{G(t,1)}{\Gamma(0.4)}, &t\in (1,1.5],\\
  \dfrac{1}{2}\Big[\Big(t+1+x_{0}(0.5)^{\lambda-1}+\dfrac{G(0.5,0)}{\Gamma(0.4)}\Big)(0.5)^{\lambda-1}+\dfrac{G(1.5,1)}{\Gamma(0.4)}\Big], &t\in (1.5,2].
    \end{array}
\right.
\end{align*}
Further, for $t\in (2,2.5] $,  the solution reduces to
{\small{
\begin{align*}
 x(t)= \dfrac{1}{2}\Big[\Big(2+1+x_{0}(0.5)^{\lambda-1}+\dfrac{G(0.5,0)}{\Gamma(0.4)}\Big)(0.5)^{\lambda-1}+\dfrac{G(1.5,1)}{\Gamma(0.4)}\Big](t-2)^{\lambda-1}+\dfrac{G(t,2)}{\Gamma(0.4)}, \end{align*}}}
  and so on. For different values of $\nu$, the value of $\lambda$ varies. For $\mu=0.4$ and $\nu=1$ the solutions are same as given in {\rm{\cite{Mittag-Impulsive-inst-noninst}}}. Also for different values of $\lambda$ and $x_{0}=1$, the graph given below provide a clear idea regarding the solution.
\begin{figure}[H]%
     \centering
     \begin{subfigure}{0.45\textwidth}
         \centering
         \includegraphics[scale=.65]{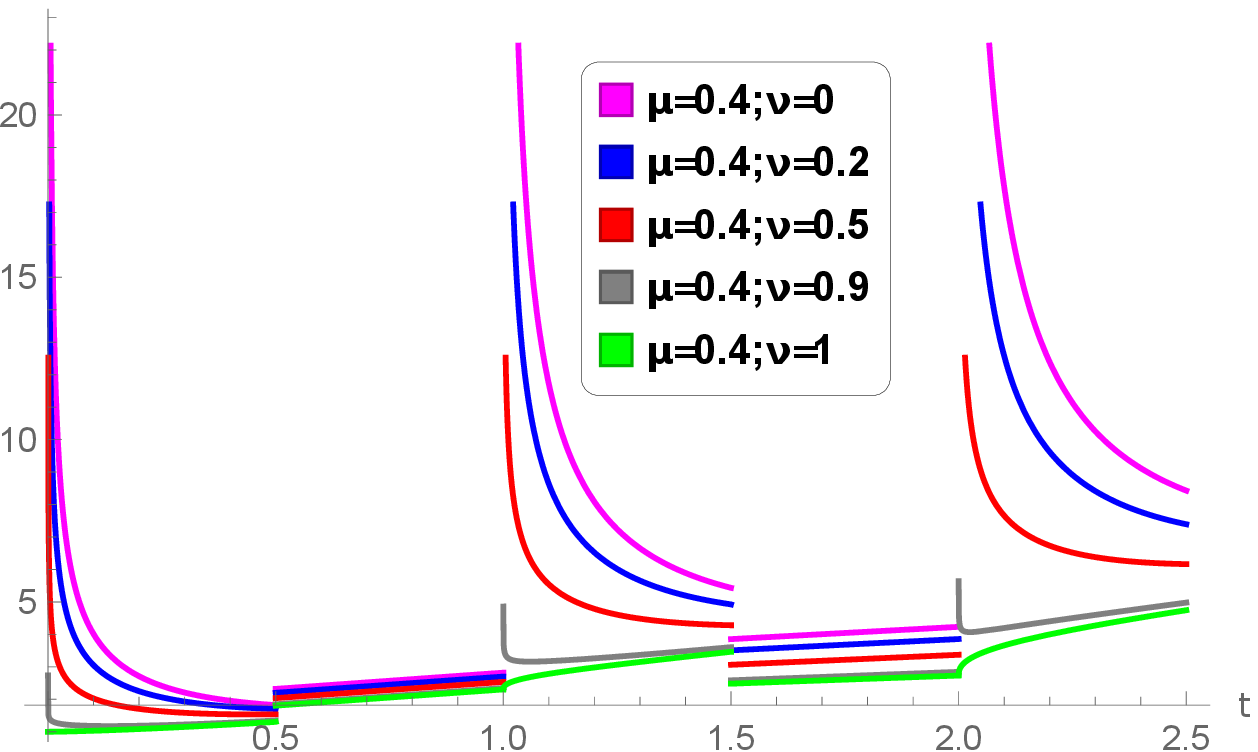}
         \caption{With $\mu=0.4$ and various values of $\nu$. }
         \label{fig1a-Stb}
     \end{subfigure}
     \hfill
     \begin{subfigure}{0.45\textwidth}
         \centering
         \includegraphics[width=\textwidth]{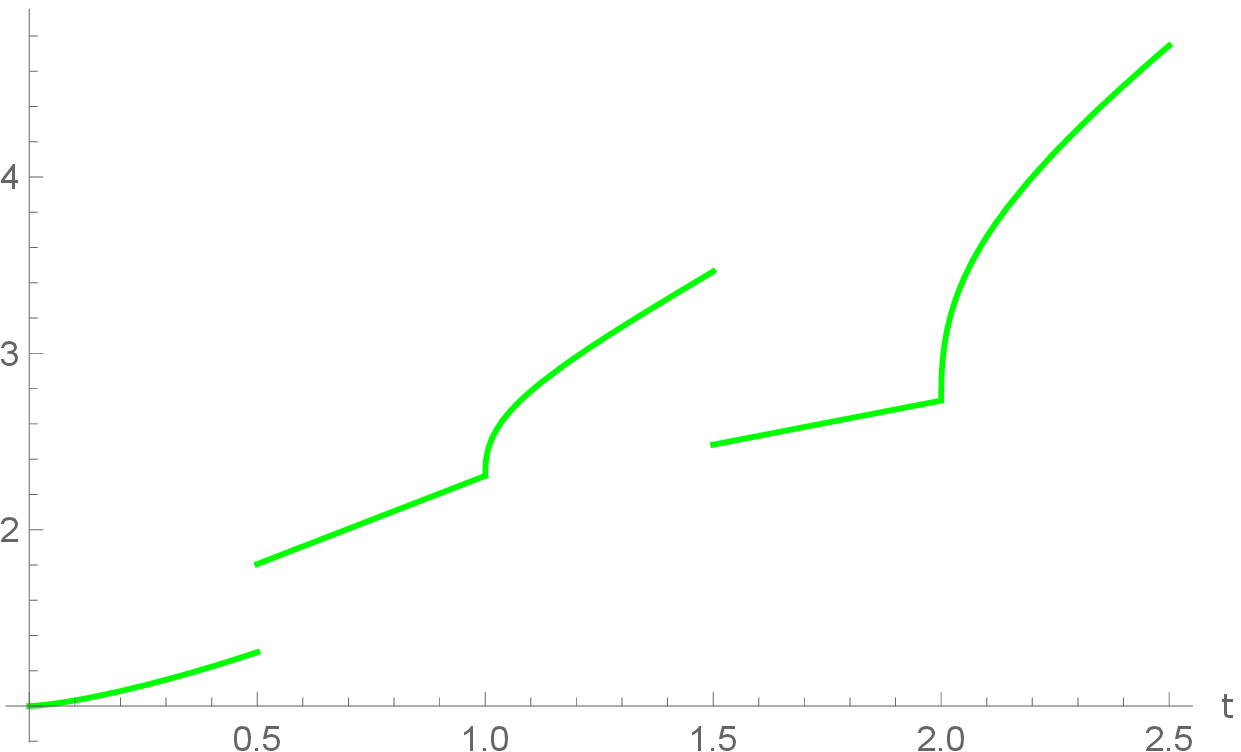}
         \caption{With $\mu=0.4$ and  $\nu=1$ }
         \label{fig1b-Stb}
     \end{subfigure}
        \caption{Caputo as a special case of Hilfer order derivative.}
        \label{fig1-Stb}
\end{figure}
\end{example}
\subsection{For Instantaneous Impulsive System}
For the instantaneous impulsive system with Hilfer fractional order, consider the sequence $\{t_{i}\}_{i=1}^{\infty}$  with $0\leq t_{0}<t_{1}<t_{2}<\ldots<t_{i+1}$ for $i=1,2,\ldots $ and $\displaystyle\lim_{i\rightarrow \infty}t_{i}=\infty$. The IVP with instantaneous impulses and Hilfer fractional derivative is given by
 \begin{align}
\label{stability-IHrFrDE-1}
{}_{t_{i}}D^{\mu,\nu}_{t}x(t)&=g(t,x) \enspace t\in(t_{i},t_{i+1}],\enspace i=0,1,\ldots ,
\end{align}
with weighted impulsive and initial condition,
\begin{eqnarray*}
%\label{stability-NIImpulse-1}
\left\{
  \begin{array}{ll}
   x(t)=\psi_{i}(t,x(t_{i}-0)),\enspace t=t_{i},\enspace i=1,2,\ldots ,\\
\lim_{t\rightarrow t_{i}}\big((t-t_{i})^{1-\lambda}x(t)\big)=\psi_{i}(t_{i},x(t_{i}-0)),\enspace i=1,2,\ldots, \\
\lim_{t\rightarrow t_{0}}\big((t-t_{0})^{1-\lambda}x(t)\big)=x_{0}.
  \end{array}
\right.
\end{eqnarray*}
Here $x_{0}\in \mathbb{R}^{n}$, $g:\cup_{i=0}^{\infty}[t_{i},t_{i+1}]\times \mathbb{R}^{n}\rightarrow \mathbb{R}^{n}$ and $\psi_{i}:[t_{0},\infty)\times \mathbb{R}^{n}\rightarrow \mathbb{R}^{n}$ for ($i=1,2,\ldots $) \\
The solution of the IVP of Hilfer fractional instantaneous differential system  \eqref{stability-IHrFrDE-1} with variable initial condition is given by
\newline
$x(t)=x(t;t_{0},x_{0})$
\begin{align*}
    =&\left\{
  \begin{array}{lll}
   X_{i}(t), & t\in (t_{i},t_{i+1}],\enspace & i=0,1,\ldots\nonumber\\
  \psi_{i}(t,X_{i}(t_{i}-0)), & t=t_{i},\enspace & i=1,2,\ldots
 \end{array}
\right.
\end{align*}
Here $X_{0}(t)$ for $t\in (t_{0},p_{0}]$ is the solution of the IVP for Hilfer fractional system \eqref{eqn:Hilfer-mittag-imp-gen} with $\tau=t_{0}$ , $\widehat{x}_{0}=x_{0}$ and for $X_{i}(t)$ for $t\in (t_{i},t_{i+1}]$, we have $\widehat{x}_{0}=\psi_{i}(t_{i},X_{i}(t_{i}-0))$.\\
In general, the solution $x(t)$, $t\geq t_{0}$ satisfies the integral system given below.
\begin{align*}
%\label{eqn:Hilfer-mittag-inst-soln}
x(t) =\left\{
  \begin{array}{ll}
\psi_{i}(t,x(t_{i}-0)),\enspace t=t_{i},\enspace i=1,2,\ldots ,\\
   \psi_{i}(t_{i},x(t_{i}-0))(t-t_{i})^{\lambda-1}\\
   +\dfrac{1}{\Gamma(\mu)}\displaystyle\int_{t_{i}}^{t}(t-s)^{\mu-1}g(s,x(s))ds, \enspace t\in(t_{i},t_{i+1}],\enspace i=0,1,2,\ldots, \\
   \psi_{0}(t,x(t_{0}-0))=x_{0}.
       \end{array}
\right.
\end{align*}
Similar to Definition \ref{def:stability-NHrFDE}, the following definition is proposed for the zero solution of an instantaneous system with Hilfer fractional derivative to be Mittag-Leffler stable.
\begin{definition}
\label{def:stability-IHrFDE}
    The zero solution of instantaneous impulsive system \eqref{stability-IHrFrDE-1}  with Hilfer fractional derivative of order $0<\mu<1$ and type $0\leq\nu\leq 1$ is called Mittag-Leffler stable, if there exist positive constants $a,b,h,\gamma$  such that for any initial time $t_{0}\in [0,p_{0})\bigcup_{i=1}^{\infty}[t_{i},p_{i})$ the following inequality holds.
\begin{align*}
\|x(t)\|\leq
\left\{
  \begin{array}{ll}
  h\|x_{0}\|^{b}\Big[\Big(\displaystyle\prod_{l=0}^{i-1}(t_{l+1}-t_{l})^{\lambda-1}E_{\mu,\nu}\big(-\gamma(t_{l+1}-t_{l})^{\mu}\big)\Big)(t-t_{i})^{\lambda-1}\\
  E_{\mu,\nu}\big(-\gamma(t-t_{i})^{\mu}\big)\Big]^{\frac{1}{a}},\enspace t\in [t_{i},t_{i+1}],\enspace i=0,1,\ldots\\
   \end{array}
\right.
\end{align*}
\end{definition}

The following lemma will be used in the proof of the main theorem to study the generalized Mittag-Leffler stability for both the impulsive systems.
\begin{lemma}
\label{lem:Hilfer-mittag-impulsive-both}
Assume that
\begin{enumerate}[\rm(1)]
\item
$g(t,0)=0$ for $t\geq 0$.
\item
$V(t,x)$ be a continuously differentiable function defined by
\begin{align*}
V(t,x):\mathbb{R}_{+}\times \Lambda \rightarrow \mathbb{R}_{+},\enspace \Lambda\subset \mathbb{R}^{n},\enspace 0\in \Lambda.
\end{align*}
\item
$V(t,x)$ is locally Lipschitz with respect to the second variable $x$.
\item
$V(t,0)=0$ for $t\in \mathbb{R}_{+}$.
\item
\begin{enumerate}[\rm(i)]
\item
$\alpha_{1}\|x\|^{a} \leq V(t,x) \leq \alpha_{2}\|x\|^{ab}$, for $t\geq \tau$, $x\in \Lambda$.
\item
${}_{\tau}D_{t}^{\mu,\nu}V(t,x(t))\leq -\alpha_{3}\|x(t)\|^{ab}$, for $t\in [\tau, p_{m}]$.
\end{enumerate}
\end{enumerate}
\noindent hold for $\tau \in [t_{m},p_{m})$, $m\geq 0$ , $m$ is an integer, $\mu\in (0,1)$, $\nu \in [0,1]$, $\alpha$, $\alpha_{2}$, $\alpha_{3}$ $a$, $b$,  are arbitrary positive constants, $\widehat{x}_{0}\in \Lambda$ and $x(t)=x(t;\tau,\widehat{x}_{0})\in C_{1-\lambda}^{\lambda}\big([\tau,p_{m}],\Lambda\big)$ is a solution of Hilfer fractional impulsive differential system \eqref{eqn:Hilfer-mittag-imp-gen}. Then
\begin{align*}
V(t,x(t))\leq \big[{}_{\tau}I_{t}^{1-\lambda}V(\tau,x(\tau))\big](t-\tau)^{\lambda-1}E_{\mu,\lambda}\big(\dfrac{-\alpha_{3}}{\alpha_{3}}(t-\tau)^\mu\big),
\enspace t\in [\tau,p_{m}]
\end{align*}
and
\begin{align*}
\|x(t;\tau,\widehat{x}_{0})\|\leq
\|\widehat{x}_{0}\|^{b}\sqrt[a]{h(t-\tau)^{\lambda-1}E_{\mu,\lambda}\big(-\gamma(t-\tau)^\mu\big)}, \enspace t\in [\tau,p_{m}]
\end{align*}
 where $h>0$ for $x(\tau)\neq 0$  and $h=0$ holds if, and only, $x(\tau)=0$.
\end{lemma}
\begin{proof}
From the conditions 5-(i) and 5-(ii)  it follows, respectively, that
\begin{align*}
\dfrac{V(t,x(t))}{\alpha_{2}}\leq \|x\|^{ab} \quad \mbox{and} \quad \dfrac{{}_{\tau}D_{t}^{\mu,\nu}V(t,x(t))}{-\alpha_{3}}\geq \|x\|^{ab}.
\end{align*}
Combining both the above inequalities gives
\begin{align*}
{}_{\tau}D_{t}^{\mu,\nu}V(t,x(t))\leq \dfrac{-\alpha_{3}}{\alpha_{2}}V(t,x(t)), \enspace t\in[\tau,p_{m}].
\end{align*}
There exists a function $W(t)\in C([\tau,p_{m}],\mathbb{R})$ such that
\begin{align*}
%\label{eqn:Hilfer-mittag-lap}
{}_{\tau}D_{t}^{\mu,\nu}V(t,x(t))+W(t)= \dfrac{-\alpha_{3}}{\alpha_{2}}V(t,x(t)), \enspace t\in[\tau,p_{m}].
\end{align*}
Taking the Laplace transform of the above system for $t\in [\tau,p_{m}]$ gives
\begin{align*}
s^{\mu}V(s)-s^{\nu(\mu-1)}\big[{}_{\tau}I_{t}^{1-\lambda}V(\tau,x(\tau))\big]+W(s)=\dfrac{-\alpha_{3}}{\alpha_{2}}V(s),
\end{align*}
where $V(s)=L\big[V(t,x(t))\big]$, $W(s)=L[W(t)]$. Further simplification leads to
\begin{align*}
V(s)=\dfrac{s^{\nu(\mu-1)}\big[{}_{\tau}I_{t}^{1-\lambda}V(\tau,x(\tau))\big]}{s^{\mu}+\frac{\alpha_{3}}{\alpha_{2}}}
            -\dfrac{W(s)}{s^{\mu}+\frac{\alpha_{3}}{\alpha_{2}}}.
\end{align*}
If  $x(\tau)=0$, then ${}_{\tau}I_{t}^{1-\lambda}V(\tau,x(\tau))=0$ and the solution to the system \eqref{eqn:Hilfer-mittag-imp-gen} becomes zero. If $\tau\neq 0$ then, as $V(t,x(t))$ is locally Lipschitz with respect to the second term, from the existence and uniqueness theorem \cite[Theorem 3.4]{Podlubny-book} and inverse Laplace transform, a unique solution exists and is given as
\begin{align*}
V(t,x(t))=&\big[{}_{\tau}I_{t}^{1-\lambda}V(\tau,x(\tau))\big](t-\tau)^{\lambda-1}E_{\mu,\lambda}\big(\dfrac{-\alpha_{3}}{\alpha_{2}}(t-\tau)^{\mu}\big)\\
        &\quad-W(t)*\big[(t-\tau)^{\lambda-1}E_{\mu,\mu}\big(\dfrac{-\alpha_{3}}{\alpha_{2}}(t-\tau)^{\mu}\big)\big].
\end{align*}
Since $(t-\tau)^{\lambda-1}\geq 0$ and $E_{\mu,\mu}\big(\dfrac{\alpha_{3}}{\alpha_{2}}(t-\tau)^{\mu}\big) \geq 0$, it follows that
\begin{align*}
V(t,x(t))\leq \big[{}_{\tau}I_{t}^{1-\lambda}V(\tau,x(\tau))\big](t-\tau)^{\lambda-1}E_{\mu,\lambda}\big(\dfrac{-\alpha_{3}}{\alpha_{2}}(t-\tau)^{\mu}\big).
\end{align*}
However from condition  5-(i) and Remark \ref{rem:Initial},  it can be concluded that
\begin{align*}
V(t,x(t))&\leq \|\widehat{x}_{0}\|^{ab}\dfrac{\alpha_{2}}{\Gamma(\lambda)} (t-\tau)^{\lambda-1}E_{\mu,\lambda}
\big(\dfrac{-\alpha_{3}}{\alpha_{2}}(t-\tau)^{\mu}\big)\\
\Longrightarrow\quad\qquad\|x\|^{a}&\leq\|\widehat{x}_{0}\|^{ab}\dfrac{\alpha_{2}}{{\Gamma(\lambda)}\alpha_{1}}(t-\tau)^{\lambda-1}E_{\mu,\lambda}
\big(\dfrac{-\alpha_{3}}{\alpha_{2}}(t-\tau)^{\mu}\big)\\
\Longrightarrow \|x(t;\tau,\widehat{x}_{0})\|&\leq \|\widehat{x}_{0}\|^{b}\sqrt[a]{h(t-\tau)^{\lambda-1}E_{\mu,\lambda}\left(-\gamma(t-\tau)^\mu\right)}
\end{align*}
with $h=\dfrac{\alpha_{2}}{{\Gamma(\lambda)}\alpha_{1}}$. This completes the proof of the lemma.
 \end{proof}
\section{Mittag-Leffler Stability of the Hilfer Fractional Differential System with Non-Instantaneous Impulses}
\label{Stability-Mittag-instant}
The main theorem that provides certain sufficient conditions for the Mittag-Leffler stability of the Hilfer fractional non-instantaneous differential
equations is given in this section. The following conditions are assumed to guarantee the existence of the solution $x(t;t_{0},x_{0})$ of the IVP \eqref{stability-NIHrFrDE-1}.
\begin{condition}
 The function $g \in C\big([0,p_{0}]\cup_{i=1}^{\infty}[t_{i},p_{i}]\times \mathbb{R}^{n},\mathbb{R}^{n}\big)$  for
$t\in (0,t_{1})\cup _{i=1}^{\infty}[t_{i},p_{i}] $ with $g(t,0)\equiv 0$ is such that for any initial point
$(\widehat{t}_{0},\widehat{x}_{0})\in [0,p_{0})\cup _{i=1}^{\infty}[t_{i},p_{i}]\times \mathbb{R}^{n}$, the IVP for general Hilfer fractional differential system \eqref{eqn:Hilfer-mittag-imp-gen} with $\tau=\widehat{t}_{0}$ has a solution
$x(t;\widehat{t}_{0},\widehat{x}_{0})\in C_{1-\lambda}^{\lambda}\big([\widehat{t}_{0},p_{m}],\mathbb{R}^{n}\big)$, where $m=min\{i:\widehat{t}_{0}<p_{k}\}$.
\end{condition}
\begin{condition}
 The function $\phi_{i}:[p_{i},t_{i+1}]\times \mathbb{R}^{n}\times \mathbb{R}^{n}\rightarrow \mathbb{R}^{n}$ for any $i=0,1,2,\ldots$ are such that, the system $x=\phi_{i}(t,x,y)$ has a unique solution $x=\zeta_{i}(t,y)$, $t\in[p_{i},t_{i+1}]$. The function $\zeta_{i}$ is defined as
$\zeta_{i}\in C\big([p_{i},t_{i+1}]\times \mathbb{R}^{n},\mathbb{R}^{n}\big)$, with $\zeta_{i}(t,0)\equiv 0$ for $t\in [p_{i},t_{i+1}]$, $i=0,1,2,\ldots$.
\end{condition}
The following theorem provides the condition for the zero solution of Hilfer fractional system with non-instantaneous impulses to satisfy the Mittag-Leffler stable condition:
\begin{theorem}
\label{thm:Stability-Th-1}
Let the assumed conditions 1 and 2 hold. $\Lambda \in \mathbb{R}^{n}$; $0\in \Lambda$. Further let the Lyapunov function $V(t,x)$ be continuously differentiable which is defined by
\begin{align*}
V(t,x):\mathbb{R}_{+}\times \Lambda\rightarrow \mathbb{R}_{+}
\end{align*}
and locally Lipschitz with respect to the second variable along with $V(t,0)=0$ for $t\geq 0$, such that
\begin{enumerate}[\rm(1)]
\item
For $t\geq 0$, $ x\in \mathbb{R}^{n}$,
\begin{align*}
%\label{eqn:Hilfer-mittag-Theorem1-cdn1}
\alpha_{1} \|x\|^{a}\leq V(t,x)\leq \alpha_{2}\|x\|^{ab},
\end{align*}
where $\alpha_{1}$, $\alpha_{2}$, $a$, $b$ are positive constants, with $\alpha_{2}\leq 1$.
\item
For any $\tau \in [0,p_{0})\cup_{k=0}^{\infty}[t_{i},p_{i}] $ and any solution $x(t)\in C_{1-\lambda}^{\lambda}([\tau,p_{m}],\mathbb{R}^{n})$ of fractional system \eqref{eqn:Hilfer-mittag-imp-gen}, the inequality
\begin{align*}
%\label{eqn:Hilfer-mittag-Theorem1-cdn2}
{}_{\tau}D_{t}^{\mu,\nu}V(t,x(t))\leq \alpha_{3} \|x(t)\|^{ab}, \enspace t\in (\tau,p_{m}]
\end{align*}
holds, where $m=min\{i:\tau<p_{i}\}$, $\mu\in (0,1)$, $\nu \in [0,1]$, $\lambda \in (0,1]$, and $\alpha_{3}>0$.
\item
For any $i=0,1,2,\ldots$, the inequality
\begin{align*}
%\label{eqn:Hilfer-mittag-Theorem1-cdn3}
V(t,\zeta_{i}(t,x))\leq \alpha_{4}\|x\|^{a}, \enspace \mbox{for}, \enspace t\in (p_{i},t_{i+1}], \enspace x\in \mathbb{R}^{n}
\end{align*}
holds, where, $\alpha_{4}$ is a positive constant such that $\alpha_{4} \leq \alpha_{1}$.
\end{enumerate}
Then the zero solution of Hilfer fractional differential system \eqref{stability-NIHrFrDE-1} is generalized Mittag-Leffler stable concerning
non-instantaneous impulses.
\end{theorem}
\begin{proof}
Let the arbitrary initial time be $t_{0}$, such that $t_{0}\in [0,p_{0})\cup _{i=1}^{\infty}[t_{i},p_{i}]$. With no loss of generality, let the
initial time be assumed as $t_{0} \in [0,p_{0})$. For the arbitrary initial point $x_{0}\in \mathbb{R}^{n}$, the solution of Hilfer fractional
impulsive system \eqref{stability-NIHrFrDE-1} is considered as $x(t;t_{0},x_{0})$. The stability is proved by the method of induction in the following steps.

\noindent \underline{Step 1.} For the interval $t \in [t_{0},p_{0}]$:\\
The solution $X_{0}(t)$ coincides with the solution of the general Hilfer impulsive system \eqref{eqn:Hilfer-mittag-imp-gen}. Here $\tau=t_{0}$; $i=0$, $\widehat{x}_{0}=x_{0}$. \\
According to Lemma \ref{lem:Hilfer-mittag-impulsive-both}, the solution can be written as,
\begin{align*}
\|x(t;t_{0},x_{0})\|&\leq \|x_{0}\|^{b}\sqrt[a]{\dfrac{\alpha_{2}}{\Gamma(\lambda)\alpha_{1}}(t-t_{0})^{\lambda-1}E_{\mu,\lambda}\big(\dfrac{-\alpha_{3}}{\alpha_{2}}(t-t_{0})^\mu\big)}.
 \end{align*}
Since $\beta\leq 1$, we have
\begin{align}
\label{eqn:Hilfer-mittag-non-inst-solu-step1}
\|x(t;t_{0},x_{0})\|\leq \|x_{0}\|^{b}\sqrt[a]{\dfrac{\alpha_{2}}{\Gamma(\lambda)\alpha_{1}}(t-t_{0})^{\lambda-1}E_{\mu,\lambda}\left(-\alpha_{3}(t-t_{0})^\mu\right)}
\end{align}
\noindent \underline{Step 2.} For the interval $t \in (p_{0},t_{1}]$:
\newline
From condition 1 and 2, for $t=p_{0}-0$, it follows that
\begin{align*}
\alpha_{1}\|x(t;t_{0},x_{0})\|^{a}&\leq V(t,x(t;t_{0},x_{0}))=V(t,x(p_{0}-0;t_{0},x_{0}))\\
&\leq \alpha_{4}\|x(p_{0}-0;t_{0},x_{0})\|^{a}.
\end{align*}
From \eqref{eqn:Hilfer-mittag-non-inst-solu-step1}, the above inequality reduces to,
\begin{align}
\label{eqn:Hilfer-mittag-non-inst-solu-step2}
\alpha_{1}\|x(t;t_{0},x_{0})\|^{a}&\leq \alpha_{4}\|x_{0}\|^{ab}\dfrac{\alpha_{2}}{\Gamma(\lambda)\alpha_{1}}(t-t_{0})^{\lambda-1}E_{\mu,\lambda}\left(-\alpha_{3}(p_{0}-t_{0})^\mu\right)\nonumber\\
\Longrightarrow \qquad \|x(t;t_{0},x_{0})\|&\leq \|x_{0}\|^{b}\sqrt[a]{\dfrac{\alpha_{2}}{\Gamma(\lambda)\alpha_{1}}(t-t_{0})^{\lambda-1}E_{\mu,\lambda}\left(-\alpha_{3}(p_{0}-t_{0})^\mu\right)}, \enspace t\in (p_{0},t_{1}].
\end{align}
\noindent \underline{Step 3.} For the interval $t \in (t_{1},p_{1}]$:
\newline
 The solution $X_{1}(t)=x(t;t_{0},x_{0})$ of the system \eqref{stability-NIHrFrDE-1} coincides with the solution of IVP \eqref{eqn:Hilfer-mittag-imp-gen} for $\tau=t_{1}$, $i=1$ and $\widehat{x}_{0}=x(t_{1};t_{0},x_{0})$. As the problem considered is changeable lower bound, for $\tau=t_{1}$, $\widehat{x}_{0}=X_{1}(t_{1})$, the inequality can be written as,
 \begin{align*}
 V(t,X_{1}(t))&\leq  [{}_{t_{1}}I_{t}^{1-\lambda}V(t_{1},X_{1}(t_{1}))](t-t_{1})^{\lambda-1}E_{\mu,\lambda}\left(-\alpha_{3}(t-t_{1})^{\mu}\right)\\
 &= [{}_{t_{1}}I_{t}^{1-\lambda} V(t_{1},x(t_{1};t_{0},x_{0}))](t-t_{1})^{\lambda-1}E_{\mu,\lambda}\left(-\alpha_{3}(t-t_{1})^{\mu}\right)\\
  &= [{}_{t_{1}}I_{t}^{1-\lambda} V(t_{1},\zeta_{0}(t_{1},x(p_{0}-0;t_{0},x_{0}))(t-t_{1})^{\lambda-1}E_{\mu,\lambda}\left(-\alpha_{3}(t-t_{1})^{\mu}\right)\\
  &\leq \dfrac{\alpha_{4}}{\Gamma(\lambda)}\|x(p_{0}-0;t_{0},x_{0})\|^{a}(t-t_{1})^{\lambda-1}E_{\mu,\lambda}\left(-\alpha_{3}(t-t_{1})^{\mu}\right).
 \end{align*}
 From Condition 1 and bound \eqref{eqn:Hilfer-mittag-non-inst-solu-step2} the below given inequality can be derived.
  \begin{align*}
 \alpha_{1}\|x(t;t_{0},x_{0})\|^{a}\leq &\enspace\dfrac{\alpha_{4}}{\Gamma(\lambda)}\|x_{0}\|^{ab}\dfrac{\alpha_{2}}{\Gamma(\lambda)\alpha_{1}}(p_{0}-t_{0})^{\lambda-1}E_{\mu,\lambda}
 \left(-\alpha_{3}(p_{0}-t_{0})^\mu\right)\\
&\qquad (t-t_{1})^{\lambda-1}E_{\mu,\lambda}\left(-\alpha_{3}(t-t_{1})^{\mu}\right).
\end{align*}
Thus from the above inequality it can be concluded that
{\small{
\begin{align*}
%\label{eqn:Hilfer-mittag-non-inst-solu-step3}
\noindent &\|x(t;t_{0},x_{0})\| \nonumber\\
\leq &\, \|x_{0}\|^{b}\sqrt[a]{\dfrac{\alpha_{2}}{(\Gamma(\lambda))^{2}\alpha_{1}}(p_{0}-t_{0})^{\lambda-1}E_{\mu,\lambda}
 \big(-\alpha_{3}(p_{0}-t_{0})^\mu\big)(t-t_{1})^{\lambda-1}E_{\mu,\lambda}\big(-\alpha_{3}(t-t_{1})^{\mu}\big)}
 \end{align*}}}
\noindent \underline{Step 4.} For the interval $t \in (p_{1},t_{2}]$:
\newline
An inequality can be derived using the condition 1 and 2, in a similar way, which is given as follows.
{\small{
 \begin{align*}
%\label{eqn:Hilfer-mittag-non-inst-solu-step4}
\noindent\Longrightarrow &\|x(t;t_{0},x_{0})\| \nonumber\\
\leq & \|x_{0}\|^{b}\sqrt[a]{\dfrac{\alpha_{2}}{(\Gamma(\lambda))^{2}\alpha_{1}}(p_{0}-t_{0})^{\lambda-1}E_{\mu,\lambda}
 \big(-\alpha_{3}(p_{0}-t_{0})^\mu\big)(p_{1}-t_{1})^{\lambda-1}E_{\mu,\lambda}\big(-\alpha_{3}(p_{1}-t_{1})^{\mu}\big)}.
 \end{align*}}}
Extending this procedure for further intervals confirms that the zero solution of the given system \eqref{stability-NIHrFrDE-1} is Mittag-Leffler stable with $h=\sqrt[a]{\dfrac{\alpha_{2}}{(\Gamma(\lambda))^{i+1}\alpha_{1}}}$ and $\alpha_{3}=\gamma$.
\end{proof}
\section{Mittag-Leffler Stability of the Hilfer Fractional Differential System with Instantaneous Impulses}
\label{Stability-Mittag-non-instant}
The main theorem that ascertains certain sufficient conditions for the Mittag-Leffler stability of the Hilfer fractional instantaneous differential
system is given in this section. The following conditions are assumed to guarantee the existence of solution $x(t;t_{0},x_{0})$ of the IVP  \eqref{stability-IHrFrDE-1}.
\begin{condition}
 The function $g \in C\big([0,\infty)/\{t_{i}\}\times \mathbb{R}^{n},\mathbb{R}^{n}\big)$  for
$t\neq t_{i} $ with $g(t,0)\equiv 0$ is such that for any initial point
$(\widehat{t}_{0},\widehat{x}_{0})\in [0,\infty)/\{t_{i}\}\times \mathbb{R}^{n}$, the IVP for general Hilfer fractional differential system \eqref{eqn:Hilfer-mittag-imp-gen} with $\tau=\widehat{t}_{0}$ has a solution
$x(t;\widehat{t}_{0},\widehat{x}_{0})\in C_{1-\lambda}^{\lambda}\big([\widehat{t}_{0},t_{m}],\mathbb{R}^{n}\big)$, where $m=min\{l:\widehat{t}_{0}<t_{i}\}$.
\end{condition}
\begin{condition}
 The function $\psi_{i}$ is defined as $\psi_{i}:[t_{0},\infty)\times \mathbb{R}^{n}\rightarrow \mathbb{R}^{n}$  with $\psi_{i}(t_{i},0)\equiv 0$ for any $i=1,2,\ldots$ .
 \end{condition}
The following theorem provides the condition for the zero solution of Hilfer fractional system with non-instantaneous impulses to satisfy the Mittag-Leffler stable condition.
\begin{theorem}
Let the assumed conditions 3 and 4 hold. $\Lambda\in \mathbb{R}^{n}$; $0\in \Lambda$. Further, let the Lyapunov function $V(t,x)$ be continuously differentiable which is defined by
\begin{align*}
V(t,x):\mathbb{R}_{+}\times \Lambda\rightarrow \mathbb{R_{+}}
\end{align*}
and locally Lipschitz with respect to the second variable along with $V(t,0)=0$ for $t\geq 0$, such that
\begin{enumerate}[\rm(1)]
\item
For $ t\geq 0$ and $x\in \Lambda$,
\begin{align*}
%\label{eqn:Hilfer-mittag-Theorem2-cdn1}
\alpha_{1} \|x\|^{a}\leq V(t,x)\leq \alpha_{2}\|x\|^{ab},
\end{align*}
where $\alpha_{1}$, $\alpha_{2}$, $a$, $b$ are positive constants, with $\alpha_{2}\leq 1$.
\item
For any $\tau \neq t_{i} $, $i=1,2,\ldots$ and for the solution $x(t)\in C_{1-\lambda}^{\lambda}([\tau,t_{m}],\mathbb{R}^{n})$ of fractional system \eqref{eqn:Hilfer-mittag-imp-gen}, the inequality
\begin{align*}
%\label{eqn:Hilfer-mittag-Theorem2-cdn2}
{}_{\tau}D_{t}^{\mu,\nu}V(t,x(t))\leq \alpha_{3} \|x(t)\|^{ab}, \enspace t\in (\tau,t_{m}]
\end{align*}
holds, where $m=min\{i:\tau<t_{i}\}$, $\mu\in (0,1)$, $\nu \in [0,1]$, $\lambda \in (0,1]$ and $\gamma >0$.
\item
For any $i=1,2,\ldots$, the inequality
\begin{align*}
%\label{eqn:Hilfer-mittag-Theorem2-cdn3}
V(t_{i},\psi_{i}(t,x))\leq \alpha_{4}\|x\|^{a}, \enspace x\in \mathbb{R}^{n}
\end{align*}
holds, where, $\alpha_{4}$ is a positive constant such that $\alpha_{4}\leq \alpha_{1}$.
\end{enumerate}
Then the zero solution of Hilfer fractional differential system \eqref{stability-IHrFrDE-1} is generalized Mittag-Leffler stable with respect to instantaneous impulses.
\end{theorem}
\begin{proof}
Let the arbitrary initial time with no loss of generality be assumed as $t_{0}\in[0,t_{1})$. For the arbitrary initial point $x_{0}\in \mathbb{R}^{n}$, with the initial time $t_{0}$, the solution is given by $x(t;t_{0},x_{0})$. As in Theorem \ref{thm:Stability-Th-1}, the proof of this theorem can be carried out interval by interval and using induction it can be extended to a general interval.
\end{proof}
\section{concluding remarks}
Mittag-Leffler stability condition for systems with both instantaneous impulses and non-instantaneous impulses having Hilfer fractional order is discussed in detail. By varying the value of $\nu$, we can interpolate the results on the stability of the solution of the system \eqref{stability-NIHrFrDE-1} and \eqref{stability-IHrFrDE-1} between Caputo and Riemann-Liouville fractional operators. Moreover, the mild solution of impulsive systems with Hilfer fractional derivative with changeable initial conditions has not been studied so far. Further, in most of the dynamical systems, delay plays an effective role for the loss in stability that degrades the performance of the system. In many applications, especially, in the field of communications and network exchange, time delay system play a major role (see, for example \cite{Stability-delay}). The stability analysis of impulsive systems with Hilfer fractional order with time delay can be considered as an immediate future problem based on this paper.

%
%\section*{Funding}
%This research did not receive any specific grant from funding agencies in the public, commercial, or not-for-profit sectors.

\section*{Authors' contribution}
The authors contributed equally to this article.

\end{document}